\documentclass[11pt,reqno]{amsart}

\usepackage{amsmath, amsfonts, amsthm, amssymb}
\usepackage{showkeys}
\newtheorem{Theorem}{Theorem}[section]
\newtheorem{Lemma}{Lemma}[section]
\newtheorem{Proposition}{Proposition}[section]
\newtheorem{Corollary}{Corollary}[section]

\theoremstyle{definition}

\theoremstyle{remark}
\newtheorem{Remark}{Remark}[section]

\numberwithin{equation}{section}

\newcommand{\eps}{\varepsilon}

\newcommand{\R}{{\mathbb R}}

\newcommand{\T}{{\mathcal T}}
\def\f{\frac}

\def\hf1{^\f{1}{1-\xi^2}}

\def\be{\begin{equation}}
\def\en{\end{equation}}
\def\bs{\begin{split}}
\def\es{\end{split}}

\author{Didier Bresch}
\address{LAMA CNRS UMR5127, Univ. Savoie Mont-Blanc,
73376 Le Bourget du Lac, France.}
\email{didier.bresch@univ-smb.fr}

\author{Pierre-Emmanuel Jabin}
\address{CSCAMM and departement of Mathematics, Univ. of Maryland, College Park, MD, USA}
\email{pjabin@cscamm.umd.edu}

\author{Zhenfu Wang}
\address{Department of Mathematics, Univ. of Pennsylvania, Phyladelphia, PA USA}
\email{zwang423@math.upenn.edu}

\title[Modulated Free Energy and Mean Field Limit]
{Modulated Free Energy and Mean Field Limit}
\date{\today}

\begin{document}

\begin{abstract}
This is the document corresponding to the talk the first author gave at IH\'ES for the Laurent Schwartz seminar on November 19, 2019.  It concerns our recent introduction of a modulated free energy
in mean-field theory in  \cite{BrJaWa}. This  physical object may be seen as a combination of the modulated potential energy introduced by S. Serfaty [See \it Proc. Int. Cong. Math. \rm  (2018)] and of  the relative entropy introduced in mean field limit theory by P.--E. Jabin, Z. Wang [See \it Inventiones \rm  2018].  It allows to obtain, for the first time, a convergence rate in the mean field limit for Riesz and Coulomb repulsive kernels in the presence of viscosity using the estimates in \cite{Du} and \cite{Se1}. The main objective in this paper is to  explain how it is possible to cover more general repulsive kernels through a Fourier transform approach as announced in \cite{BrJaWa} first in the case $\sigma_N \to 0$ when $N\to +\infty$ and then if $\sigma>0$ is fixed. Then we end the paper with comments on the particle approximation 
of the Patlak-Keller-Segel system which is associated to an attractive kernel and refer to
 [{\it C.R. Acad Science Paris } 357, Issue 9, (2019), 708--720] by the authors for more details.
\end{abstract}

\maketitle

\section{Introduction} This paper concerns interaction particles system and quantitative estimates in mean field limit theory in the spirit of \cite{Se1}, \cite{Du} and \cite{JaWa}.   Namely we consider $N$ particles, identical and interacting two by two through a kernel $K$. For simplicity, we consider periodic boundary conditions and we assume the  position of the i-th particle  $X_i(t) \in \Pi^d$  evolves through as follows
$$dX_i = \frac{1}{N} \sum_{j\not=i} K(X_i-X_j) dt + \sqrt{2\sigma} dW_i$$
for $N$ independent Brownian motions $W_i$ with
a gradient flow hypothesis $K= - \nabla V$ where $V\in L^1(\Pi^d)$ will be discussed later-on:
For possible vanishing viscosity with respect to $N$ namely  $\sigma_N\to 0$ when $N\to +\infty$ we consider singular repulsive kernels with some pointwise and Fourier controls (including Riesz and Coulomb Kernels) and for $\sigma>0$ we consider more general singular repulsive kernels with Fourier control and then we conclude with some comments on the Patlak-Keller-Segel attractive kernel. The main objective in this document is to give some details in the repulsive case to complete the note written by the authors [{\it C.R. Acad Sciences} 357, Issue 9, (2019), 708--720] which focused on the attractive Patlak-Keller-Segel kernel. A full paper  is still in  progress to propose a single complete document, see \cite{BrJaWa1}.
Readers interested by some reviews on mean field limit for stochastic particle systems are referred to \cite{JaWa1}, \cite{Sa} and on mean field limit for deterministic particle systems are referred to \cite{Go}, \cite{Ja}, \cite{Se1}.

\medskip
\noindent As usually, we introduce $\rho_N$ the joint law of the process $(X_1, \cdots, X_N)$ which satisfies the Liouville equation
$$\partial_t \rho_N + \sum_{i=1}^N {\rm div}_{x_i}
 \left(\rho_N \frac{1}{N} \sum_{j\not= i}^N K(x_i-x_j) \right) = \sigma \sum_{i=1}^N \Delta_{x_i} \rho_N.$$
The main objective is to obtain a rate of convergence with respect to the number of particles $N$  from $\rho_N$ to  $\bar\rho_N= \bar\rho^{\otimes N} = \Pi_{i=1}^N \bar\rho$ with
$$\partial_t \bar\rho + {\rm div} (\bar\rho u) = \sigma \Delta \bar\rho,
   \qquad u = - \nabla V \star_x \bar\rho$$
where $\bar\rho$ is a 1-particle distribution namely $\bar\rho\ge 0$ and $\int\bar\rho = 1$.  
More precisely,  the main objective is to prove (for $\sigma>0$ fixed) that there exists constants $C_{T,\bar \rho, k} >0$ and $\theta>0$ such that  
$$\|\rho_{N,k} - \Pi_{i=1}^k \bar \rho(t,x_i)\|_{L^1(\Pi^{kd})}  \le C_{T,\bar \rho, k} N^{-\theta}$$
where $\rho_{N,k}$ is the marginal of the system at a fixed rank $k$,
$$
\rho_{N,k}(t,x_1,\ldots,x_k)=\int_{\T^{(N-k)\,d}} \rho_N(t,x_1,\ldots,x_N)\,dx_{k+1}\ldots x_{N}.
$$
under assumptions of global existence of entropy-weak solutions $\rho_N$  of the
Liouville equation  and global existence of classical solution $\bar\rho$ of the limiting system. 
Due to the Csisz\'ar-Kullback-Pinsker Inequality, it suffices to control the quantity
$$\frac{1}{k} \int_{\Pi^{kd}} \rho_{N,k} \log \frac{\rho_{N,k}}{\bar\rho^{\otimes k}}$$
and therefore
$$\frac{1}{N} \int_{\Pi^{Nd}} \rho_N \log \frac{\rho_N}{\bar\rho_N}
$$
due to the inequality
$$\frac{1}{k} \int_{\Pi^{kd}} \rho_{N,k} \log \frac{\rho_{N,k}}{\bar\rho^{\otimes k}}
\le \frac{1}{N} \int_{\Pi^{Nd}} \rho_N \log \frac{\rho_N}{\bar\rho_N}.
$$
In the case $\sigma_N \to 0$ when $N\to +\infty$, the information will be related to a rescaled modulated potential energy similar to
the one introduced by S.  Serfaty.

As explained in \cite{Se} (and commented by F.  Golse during the talk) the keys of the paper by S. Serfaty is to introduce a clever truncation  of the kernel  at a length-scale $r_i$ depending on the point $i$ and equal to a minimal distance from $x_i$ to its nearest neighbors. Using such truncation  it is possible to prove that the truncated energy can be controlled by the full energy and conversely.
The next crucial result is that, even though positivity is lost, when using the $r_i$™ as truncation parameters she can still control for each time slice the expression by the modulated energy itself, up to a small error , and provided the limiting solution is regular enough. This, which is the most difficult part of her proof, uses two ingredients: the first is to re-express as a single integral in terms of a stress-energy tensor, and the second is to show that the expression  is in fact close to the same expression with truncated fields. 

In this document, we will explain how to construct an appropriate regularized kernel (Lemma \ref{lemmatruncate}) which in some sense play the role
of the appropriate truncation in \cite{Se1}. In the case $\sigma_N\to 0$, it will be used to control the contribution to the potential energy for close particles switching between the kernel and its regularization. It will also be used  coupled with a Fourier transform property to get rid of the re-expression as a single integral in terms of a stress-energy tensor. In the case $\sigma>0$ fixed the Fourier transform hypothesis may be relaxed playing with the classical convexity inequality \eqref{ineg} and large deviation type estimates.

    The paper is divided in seven sections. Section 1 is the present introduction. Section 2  is dedicated to the modulated free energy \eqref{freeenergy} (with $G_N$ and $G_{\bar\rho_N}$ defined respectively by \eqref{GN} and \eqref{G}) introduced by the authors in \cite{BrJaWa1} which allows to make the link between \cite{Se1} and \cite{JaWa}. The analysis is based on the important inequality \eqref{INEG} with \eqref{IN} 
with Gronwall Lemma and some appropriate estimates. Section 3 concerns the assumptions \eqref{hyp01}--\eqref{hyp06} on the kernels and the main results Theorem \ref{Conv} and Theorem \ref{Convgene}. We focus on repulsive kernels and provide Theorems first in the case $\sigma_N\to 0$ when $N\to +\infty$ and secondly when $\sigma$ is fixed. We provide simple comments for the attractive kernel corresponding to the Patlak-Keller-Segel system and refer to \cite{BrJaWa} for more details. Section 4 is dedicated to an important regularization lemma which will  helpful to switch between the kernel and an appropriate regularized one. Section 5 and Section 6 concern the proofs of different controls from below and above coupled with inequality \eqref{INEG} with \eqref{IN} firstly when $\sigma_N \to 0$ when $N\to +\infty$ and secondly when $\sigma$ is fixed. Section 7 concludes the proofs of  Theorem \ref{Conv} and Theorem \ref{Convgene} from Gronwall arguments.  In Section 8, we provide comments on the interesting attractive case corresponding to the particle approximation of the Patlak-Keller-Segel system. We end the paper by Section 9 with some
comments and open problems.

\section{Physics provides the right mathematical Object} As firstly introduced by the authors in \cite{BrJaWa}, keeping advantage of the idea to introduce appropriate weights from \cite{BrJa}, we define the following modulated free energy
\begin{equation}
\label{freeenergy}
E\Bigl(\frac{\rho_N}{G_N}\vert \frac{\bar\rho_N}{G_{\bar\rho_N}}\Bigr) 
= \frac{1}{N} \int_{\Pi^{dN}} \rho_N (t,X^N) \log\left(\frac{\rho_N(t,X^N)}{G_N(X^N)} 
   \frac{G_{\bar\rho_N}(t,X^N)}{\bar\rho_N(t,X^N)}\right)dX^N.
\end{equation}
where the Gibbs equilibrium $G_N$ of the system and $G_{\bar\rho_N}$ the corresponding 
distribution where the exact field is replaced by the mean field limit according to the law $\bar\rho$
are given by
\begin{equation}\label{GN}
G_N(t,X^N) = \exp \left( - \frac{1}{2N\sigma} \sum_{i\not = j} V(x_i-x_j)\right)
\end{equation}
and
\begin{equation} \label{G}
G_{\bar\rho_N}(t,X^N) = \exp\left(-\frac{1}{\sigma} \sum_{i=1}^N V\star \bar\rho(x_i)
        + \frac{N}{2\sigma} \int_{\Pi^d} V\star \bar\rho \bar\rho  \right)
\end{equation}
It may be written
$$E\Bigl(\frac{\rho_N}{G_N}\vert \frac{\bar\rho_N}{G_{\bar\rho_N}}\Bigr) 
= {\mathcal H}_N(\rho_N\vert \bar\rho_N) + {\mathcal K}_N(G_N\vert  G_{\bar\rho_N}) 
$$
where 
\begin{equation}\label{relativeentropy}
{\mathcal H}_N(\rho_N\vert \bar\rho_N) 
= \frac{1}{N} \int_{\Pi^{dN}} \rho_N(t,X^N) \log\left( \frac{\rho_N(t,X^N)}{\bar\rho_N(t,X^N)}\right) dX^N
\end{equation}
is exactly the relative entropy introduced by Jabin-Wang  and
\begin{equation}\label{modulatedenergy}
{\mathcal K}_N(G_N\vert G_{\bar\rho_N}) =
   - \frac{1}{N} \int_{\Pi^{dN}} \rho_N(t,X^N) 
     \log \left(\frac{G_N(X^N)}{G_{\bar\rho_N}(t,X^N)}\right) dX^N
\end{equation}
is the expectation of the modulated potential energy introduced by S. Serfaty multiplied
by $1/\sigma$. It is then possible to show that the free energy has the right algebraic structure for any
$V$ even. More precisely, using smoothing properties and definition of global entropy-weak solution of
the Liouville system and global classical solution of the limit system, we get the inequality
\begin{equation} \label{INEG}
\begin{split}
&E_{N}\left(\frac{\rho_N}{ G_{N}} \,\vert\;\frac{\bar\rho_N}{G_{\bar\rho_N}}\right)(t) 
+  \frac{\sigma}{N}\,\int_0^t \int_{\Pi^{d\,N}} d\rho_N\,\left|\nabla\log \frac{\rho_N}{\bar\rho_N}-\nabla\log \frac{ G_N}{G_{\bar\rho_N}}\right|^2
\\
& \hskip7cm  \le 
E_{N}\left(\frac{\rho_N}{ G_{N}} \,\vert\;\frac{\bar\rho_N}{G_{\bar\rho_N}}\right)(0) 
 + I_N
 \end{split}
 \end{equation}
 where
\begin{equation}\label{IN}
\begin{split}
& I_N =  
 -\frac{1}{2} \int_0^t\int_{\Pi^{dN}} \int_{\Pi^{2\,d}\cap \{x\neq y\}}  \nabla V(x-y) \cdot \left(\nabla\log \frac{\bar\rho}{G_{\bar \rho}}(x)
  - \nabla\log \frac{\bar\rho}{G_{\bar \rho}}(y)\right) \\
  &\hskip10cm  (d\mu_N - d\bar\rho)^{\otimes 2} d\rho_N,
  \end{split}
\end{equation}
where $\mu_N = \sum_{i=1}^N \delta(x-x_i)/N$ is the empirical measure. It is important to
note that the right-hand side is exactly the expectation of the quantity obtained by S. Serfaty
with the modulated potential energy when $\sigma = 0$  and to remark that the parasite terms 
involving ${\rm div} K$ in the work by P.--E. Jabin and Z. Wang have disappeared. 
  In order to use Gronwall Lemma, the previous expression leaves two main points in the proof namely the existence
 $\theta>0$ and $C>0$ such that we have
 
 \noindent {I) \it An upper bound of $I_N$ given by \eqref{IN}.} If the viscosity $\sigma>0$ is fixed
namely
$$ I_N \le \int_0^t E_{N}\left(\frac{\rho_N}{ G_{N}} \,\vert\;\frac{\bar\rho_N}{G_{\bar\rho_N}}\right)(s) \, ds + \frac{C}{N^\theta}$$
and  if $\sigma_N \to 0$ when $N\to+\infty$
$$
\sigma_N I_N  \le  C \sigma_N \int_0^t  {\mathcal K}_N \left(\frac{\rho_N}{ G_{N}} \,\vert\;\frac{\bar\rho_N}{G_{\bar\rho_N}}\right)(s) \, ds 
                 + \frac{C}{N^\theta}
$$
\noindent {II) \it A control from below on modulated quantities.} For $\sigma>0$ fixed,  the modulated free energy $E_N$ is almost positive or more specifically that for some constant
$$E_N\left(\frac{\rho_N}{G_N}\vert\frac{\bar\rho_N}{G_{\bar\rho_N}}\right) (t)
   \ge \frac{1}{C} {\mathcal H}_N(\rho_N\vert\bar\rho_N)(t) - \frac{C}{N^\theta}$$
  and, for $\sigma_N \to 0$ when $N\to +\infty$, the rescaled modulated potential energy is almost positive namely  
$$\sigma_N {\mathcal K}_N \ge - \frac{C}{N^\theta}.$$

\begin{Remark} Combining the relative entropy with a modulated energy has bee successively used for various limit in kinetic theory such as quasi neutral limit (see  \cite{Kw}, \cite{PuSa}) or Vlasov-Maxwell-Boltzmann to incompressible viscous EMHD  (see \cite{ArSa}).
\end{Remark}

\begin{Remark} Note the presence of modulated Fisher type information 
\begin{equation}\label{diffusive}
{\mathcal D} = 
\frac{\sigma}{N}\,\int_0^t \int_{\Pi^{d\,N}} d\rho_N\,\left|\nabla\log \frac{\rho_N}{\bar\rho_N}-\nabla\log \frac{ G_N}{G_{\bar\rho_N}}\right|^2
\end{equation}
in the inequality \eqref{IN}. 
\end{Remark}

\section{Assumptions and main results.} We will split the discussions in three parts.
The first part concerns the case with viscosity $\sigma\to 0$ when $N\to +\infty$. We show how to consider
more general singular kernels than in \cite{Se}, \cite{Du} extending their methods.
The second part concerns a fixed viscosity $\sigma>0$ using the modulated free
energy where we indicate kernels that may be considered. In the last part,
we give comments regarding the Patlak-Keller-Segel system which concerns
an attractive kernels and viscosity $2d \sigma>\lambda$ where $\lambda$ measures
the intensity of the kernel.

\medskip

\noindent {I) \bf Repulsive Kernels.} In this part, the first assumptions on $V$ are 
\begin{equation} \label{hyp01}
V(-x)=V(x),\quad V\in L^p(\Pi^d) \hbox{ for some } p>1
\end{equation}
with following Fourier sign 
\begin{equation}\label{hyp03} \hat V(\xi) \ge 0 \hbox{ for all } \xi \in \R^d
\end{equation}
and the following pointwise controls for all $x\in {\mathbb T}^d$: There exists constants 
$k$ and $C>0$ such that
\begin{equation} \label{hyp001}
 |\nabla V(x)|\leq \frac{C}{|x|^k},\quad |\nabla^2 V(x)|\leq \frac{C}{|x|^{k'}} \qquad  \hbox{ for some } k,k'>1/2
\end{equation}
and 
\begin{equation}\label{hyp02} 
|\nabla V(x)|\leq C\,\frac{V(x)}{|x|}.
\end{equation}
We also assume that
\begin{equation}\label{hyp04}\lim_{|x|\to 0} V(x)= +\infty, \qquad
 V(x) \le C\, V(y) \hbox{ for all } |y|\le 2 |x|
\end{equation}
and  
\begin{equation}\label{hyp05} 
 |\nabla_\xi \hat V(\xi)| \le \frac{C}{1+|\xi|} \left(\hat V(\xi)+   f(\sigma)\frac{1}{1+|\xi|^{d-\alpha}}\right)
  \hbox{ with }  0<\alpha<d   \hbox{ for all } \xi\in \R^d 
 \end{equation}
 where 
 \begin{equation}\label{hyp06} 
f(\sigma) = 0 \hbox{ if } \sigma \to 0 \hbox{ and } f(\sigma)=1 \hbox{ if } \sigma \hbox{ is fixed.}
 \end{equation}

\begin{Remark}
Remark that Riesz and Coulomb kernels  perturbed to get periodic kernels satisfy hypothesis \eqref{hyp01}--\eqref{hyp06}. 
Note that constraints $k,k'>1/2$ are chosen for simplicity in the argument: the results cover any $k$ and $k'$.
\end{Remark}
\medskip

\noindent {I-1)  \it  Case $\sigma_N \to 0$ and repulsive kernels.} 
   The convergence rate theorem reads as follows
  \begin{Theorem}\label{Conv}
  Assume $K=-\nabla V$ with $V$ satisfying \eqref{hyp01}--\eqref{hyp06} with $f(\sigma)=0$.
  Consider $\bar \rho$ a smooth enough solution  with $\inf\bar\rho>0$. There exists constants $C>0$ and a function $\eta(N)$ with $\eta(N)\to 0$ as $N\to \infty$ s.t. for $\bar \rho_N=\Pi_{i=1}^N \bar \rho(t, x_i)$, and for the  joint law $\rho_N$ on $\Pi^{dN}$ of any entropy-weak solution to the SDE system, 
\[\begin{split}
& \sigma_N\,\mathcal{K}_N(t)\leq e^{C_{\bar\rho}\,\|K\| \,t}\,\left(\sigma_N\,\mathcal{K}_N(t=0)+\sigma_N\,\mathcal{H}_N(t=0)+\eta(N)\right),\\
\end{split}\]
Hence if $\sigma_N\,\mathcal{K}_N^0+\sigma_N\,\mathcal{H}_N^0\leq \eta(N)$, for any fixed marginal $\rho_{N,k}$ 
      \[
W_1(\rho_{N,k},\;\Pi_{i=1}^k \bar\rho(t, x_i))\leq C_{T,\bar\rho,k}\,\eta(N),
      \]
 where $W_1$ is the Wasserstein distance.
\end{Theorem}

\smallskip
\noindent {I-2)  \it Case  $\sigma>0$ fixed with respect to $N$ and repulsive kernels.}
The convergence rate theorem reads as follows
 \begin{Theorem}\label{Convgene}
  Assume $K=-\nabla V$ with $V$ satisfying \eqref{hyp01}--\eqref{hyp06} with $f(\sigma)=1$.
  Consider $\bar \rho$ a smooth enough solution  with $\inf\bar\rho>0$. There exists  constants $C>0$ and $\theta>0$ s.t. for $\bar \rho_N=\Pi_{i=1}^N \bar \rho(t, x_i)$, and for the joint law $\rho_N$ on $\Pi^{dN}$ of any entropy-weak solution to the SDE system, 
\[\begin{split}
& H_N(t)+ |\mathcal{K}_N(t)| \leq e^{C_{\bar\rho}\,\|K\| \,t}\,\left(H_N(t=0)+ |\mathcal{K}_N(t=0)|+\frac{C}{N^\theta}\right),\\
\end{split}\]
Hence if $H_N^0+|\mathcal{K}_N^0|\leq C\,N^{-\theta}$, for any fixed marginal $\rho_{N,k}$ 
\[
\|\rho_{N,k}-\Pi_{i=1}^k \bar\rho(t, x_i)\|_{L^1(\Pi^{k\,d})}\leq C_{T,\bar\rho,k}\,N^{-\theta}.
\]
\end{Theorem}

\smallskip

\begin{Remark} In the case $\sigma>0$ fixed, it is possible to enlarge the class of kernels for instance assuming only $V\ge -C$ with
$C> 0$ and choosing $\alpha=0$ in \eqref{hyp05}.  To allow $\alpha=0$ requires a novel large deviation inequality, similar in spirit to Prop. \ref{LargeDeviation} but for singular attractive potentials.  More precisely, we can use the following proposition for which the proof is more complex and we refer to our upcoming article \cite{BrJaWa1} for it.
  There exists $\eta_0>0$, $\theta>0$ s.t. if $G(x)\leq C\,\log\frac{1}{|x|}+C$ with $C>0$ and $\eta\leq \eta_0$ then
  \[
\int_{\Pi^{d\,N}}\int_{\{x\neq y\}\cap\{|x-y|\leq\eta\}} G(x-y)\,(d\mu_N-d\bar\rho)^{\otimes2}\,d\rho_N\leq C\,\mathcal{H}_N+\frac{C}{N^\theta}.
  \]
 Note that such control is also central in the proof for the attractive Patlak-Keller-Segel kernel, see  \cite{BrJaWa} for a sketch.
\end{Remark}

\medskip

\noindent {II)  \bf An Attractive case.} This part concerns the first
quantitative estimate related to particle approximation of the Patlak-Keller-Segel system. Namely we get
the same conclusion than in Theorem \ref{Convgene}   for 
$$V=\lambda \log |x|+ perturbation$$
if $0\le \lambda<2d \sigma$  with the perturbation being regular kernel to get $V$ periodic.
We refer the readers to \cite{BrJaWa} for the explanation  of the steps in the Proof. 
Some comments  are given at the end of the present paper.

\section{An important regularization lemma}
In our approach an important Lemma will be the construction of an appropriate regularized kernel using hypothesis \eqref{hyp01}--\eqref{hyp04}.
More precisely we prove the following Lemma 
\begin{Lemma} Let $V$ satisfying \eqref{hyp01}--\eqref{hyp02}.
Then  there exists a smooth approximation $V_\varepsilon$ of $V$  and a function of $\eta(\varepsilon)$
with $\eta(\varepsilon) \to 0$ as $\varepsilon \to 0$ such that
\[
\begin{split}
& \hat V_\varepsilon \ge 0,  \\
&\|V_\varepsilon - V \|_{L^1(\Pi^d)} \le \eta(\varepsilon),\quad 
\|\mathbb{I}_{|x|\geq \delta}\,(V-V_\eps)\|_{L^1}\leq C\,\frac{\eps}{\delta^k},\quad \|\mathbb{I}_{|x|\geq \delta}\,(\nabla V-\nabla V_\eps)\|_{L^1}\leq C\,\frac{\eps}{\delta^{k'}},\\
& V_\varepsilon(x) \le V(x) + \varepsilon \hbox{ for all } x.
\end{split}
\]
\label{lemmatruncate}
\end{Lemma}
This step asks for appropriate regularization which in some sense depend on the total number of particles $N$: It uses the pointwise properties of the kernels \eqref{hyp001}--\eqref{hyp04} which includes a doubling variable property.
\begin{proof}

\noindent {\it Some preliminary controls.}
 Consider a smooth kernel $K^1$ with compact support in $B(0,1)$ and s.t. $\hat K^1\geq 0$.
Observe that for $|x|\geq 2\,\delta$ and since $|\nabla V(y)|\leq \frac{C}{|x|^{k}}$ for $|y|\geq |x|/2$,
  \[
|K^1_\delta\star V(x)-V(x)|\leq \int_{|z|\leq 1} K^1(z)\,|V(x-\delta\,z)-V(x)|\,dz\leq 
C\,\frac{\delta}{|x|^{k}}.
\]
Therefore
\begin{equation}
K^1_\delta\star V(x)\leq V(x)+C\,\delta^{1/2},\quad\forall\;|x|\geq \delta^{1/2k}.\label{|x|geqdelta}
\end{equation}
On the other hand, we also have that $K^1_\delta\star V(x)\leq C\,\delta^{-d}\,\|V\|_{L^1}$. Since $V(x)\to \infty$ as $|x|\to 0$, we also have that for some increasing function $f(\delta)$
\begin{equation}
K^1_\delta\star V(x)\leq V(x),\quad\forall\;|x|\leq f(\delta).\label{|x|leqdelta}
\end{equation}
As $k$ is chosen such that $k>1/2$ then $\delta^{1/2k}\ge 2\delta$, so we need to be more precise where $f(\delta)\leq |x|\leq \delta^{1/2k}$.

\smallskip
\noindent {\it Construction of an appropriate regularized kernel $W_\eps$.}

\noindent {\it Case $|x|\ge 2\delta$.} First of all we notice that by the doubling property, we have directly if $|x|\geq 2\,\delta$ that $|x|/2\leq |x-\delta\,z|\leq 2|x|$ for $|z|\leq 1$ and thus
\begin{equation}
K^1_\delta\star V(x)\leq C\,V(x),\quad\forall\;|x|\geq 2\,\delta.\label{|x|sim2delta}
  \end{equation}

\smallskip

\noindent {\it Case $|x|\le  2\delta$.}  For $|x|\leq 2\,\delta$, the doubling property on its own only gives that
\begin{equation}
K^1_\delta\star V(x)\leq C\,K^1_\delta\star V(0)\leq C\,\delta^{-d}\,\int_{B(0,\delta)} V(y)\,dy.\label{K*V(0)}
\end{equation}
We now define a sequence $\delta_n\to 0$ s.t.
\begin{equation}
\int_{B(0,\delta_n)} V(y)\,dy\leq C\,\int_{\delta_n/2\leq |y|\leq \delta_n} V(y)\,dy.
\label{defdelta_n}
\end{equation}
The existence of such a sequence is straightforward to show by contradiction, as otherwise, we would have for some $\bar\delta$ and all $\delta\leq \bar\delta$ that
\[
\int_{B(0,\delta/2)} V(y)\,dy\geq (C-1)\,\int_{\delta/2\leq |y|\leq \delta} V(y)\,dy.
\]
By induction, this would imply that
\begin{equation}\label{contrad}
\int_{B(0,\bar \delta/2^k)} V(y)\,dy\geq (1-1/C)^{k-1} (C-1)\int_{\bar \delta/2\leq |y|\leq \bar \delta} V(y)\,dy.
\end{equation}
Using the $L^p$ bound on $V$, we have
\[
\int_{B(0,\bar \delta/2^k)} V(y)\,dy 
\le C' \, 2^{-kp/(p-1)}
\]
this provides a contradiction with \eqref{contrad}  if $C-1$ is too large. Note by the way that if we assume some explicit rate on $V$ then we can have explicit bound on how large $\delta_n$ can be w.r.t. $\delta_{n+1}$. 

If $\delta=\delta_n$, then \eqref{defdelta_n} and \eqref{K*V(0)} together implies that
\[
K^1_\delta\star V(x)\leq C\, \delta_n^{-d} \int_{\delta_n/2\leq |y|\leq \delta_n} V(y)\,dy.
\]
This is where we use \eqref{hyp02} which implies that if $|x|\leq 2\,\delta_n$ and $\delta_n/2\leq |y|\leq \delta_n$ then $V(y)\leq C\,V(x)$. Hence eventually for $\delta=\delta_n$, we obtain the counterpart to \eqref{|x|sim2delta} and find
\begin{equation}
K_{\delta_n}\star V(x)\leq C\, V(x), \qquad\forall x.\label{roughgeneral}
\end{equation}

\bigskip

\noindent Now for every $\eps$, we are going to choose $M$ of the parameters $\delta_n$, with $M$ the integer part of $\eps$ large, and define first
\[
W_\eps=\frac{1}{M}\sum_{i=1}^M K_{\delta_{n_i}}\star V.
\]
We start by taking $\delta_{n_1}\leq\eps^2/C$ and we then define the $n_i$ recursively s.t.
\[
\delta_{n_{i+1}}\leq\min(f(\delta_{n_i}),\;\delta_{n_i}^{2k}).
\]
We of course have automatically that $\hat W_\eps\geq 0$. Moreover since $\max \delta_{n_i}\to 0$ as $\eps\to 0$, the standard approximation by convolution shows that $\|V-W_\eps\|_{L^1}\to 0$. By using \eqref{hyp001}, we also directly have that
\[
\|\mathbb{I}_{|x|\geq \delta}\,(V-W_\eps)\|_{L^1}\leq C\,\frac{\max\delta_{n_i}}{\delta^k}\leq C\,\frac{\eps}{\delta^k},
\]
and similarly
\[
\|\mathbb{I}_{|x|\geq \delta}\,(\nabla V-\nabla W_\eps)\|_{L^1}\leq C\,\frac{\eps}{\delta^{k'}}.
\]

\smallskip

It only remains to compare $W_\eps$ and $V$. For this consider any $x$, if $|x|\geq \delta_{n_1}^{1/2k}$ then \eqref{|x|geqdelta} directly implies that $W_\eps(x)\leq V(x)+\eps$. If $|x|\leq f(\delta_{n_M})$ then \eqref{|x|leqdelta} also directly implies that $W_\eps(x)\leq V(x)$.

This only leaves the case where $|x|$ is somewhere between $\delta_{n_M}$ and $\delta_{n_1}$. In that case, there exists $i$ s.t. $\delta_{n_{i+1}}\leq |x|\leq \delta_{n_i}$. By the definition of the $\delta_{n_j}$, one has that $|x|\leq f(\delta_{n_j})$ if $j<i$ and $|x|\geq \delta_{n_j}^{1/2k}$ if $j>i+1$. Using again \eqref{|x|geqdelta} and  \eqref{|x|leqdelta}, we then have
\[
K_{\delta_{n_j}}\star V(x)\leq V(x) + \eps \quad\mbox{if}\ j>i+1,\qquad K_{\delta_{n_j}}\star V(x)\leq V(x) \quad\mbox{if}\ j<i.
\]
Using \eqref{roughgeneral} for $j=i$ and $j=i+1$, we get
\[
W_\eps(x)\leq (1+2\,C/M)\,V(x)+\eps=(1+2\,C\,\eps)\,V(x)+\eps.
\]

\smallskip

\noindent {\it Definition of $V_\epsilon$ and conclusion.}
This leads to the final definition $V_\eps=W_\eps/(1+2\,C\,\eps)$ which indeed satisfies $V_\eps\leq V+\eps$, $\hat V_\eps\geq 0$ and still $\|V_\eps-V\|_{L^1}\to 0$ together with the other convergences since obviously $\|V_\eps-W_\eps\|_{L^1}\to 0$.
  \end{proof}

\bigskip

\section{Control of terms for Repulsive Kernels with vanishing viscosity $\sigma_N\to 0$}
   As a corollary of Lemma \ref{lemmatruncate}, we obtain a straightforward control on the contribution to the potential energy from close particles
\begin{Lemma}
  Under the assumptions of Lemma \ref{lemmatruncate}, there exists a function $\eta(N)$ with $\eta(N)\to 0$ as $N\to \infty$, one has that
  \[
\sigma_N\,\mathcal{K}_N(t)\geq -\eta(N).
\]
Furthermore there exists $\eta(\delta)$ with $\eta(\delta)\to 0$ as $\eta\to 0$, such that for any $\delta$
\[
\mathbb{E}\bigl(\frac{1}{N^2}\,\sum_{i\neq j} V(x_i-x_j)\,\mathbb{I}_{|x_i-x_j|\leq \delta}\bigr)\leq 4\,\sigma_N\,\mathcal{K}_N(t)+\eta(N)+\eta(\delta).
\]\label{KNgeq0}
\end{Lemma}
\begin{proof}
We  start by noticing that 
  \[
\sigma_N\,\mathcal{K}_N(t)=\frac{1}{2}\,\int_{|x-y|\geq \delta} V(x-y)\,(\mu_N-\bar\rho)^{\otimes 2}+\frac{1}{2}\,\int_{|x-y|< \delta, \;x\neq y} V(x-y)\,(\mu_N-\bar\rho)^{\otimes 2}.
\]
By using the regularity of $\bar\rho$ and the $L^p$ integrability of $V$, we may bound from below the second term in the right-hand side by
\[
\int_{|x-y|< \delta, \;x\neq y} V(x-y)\,(\mu_N-\bar\rho)^{\otimes 2}\geq \frac{1}{N^2}\,\sum_{i\neq j} V(x_i-x_j)\,\mathbb{I}_{|x_i-x_j|\leq \delta}-C\,\delta^\alpha,
\]
for some positive exponent $\alpha$. Using Lemma \ref{lemmatruncate} and more precisely the inequality $V_\eps\leq V+\eps$, we obtain that
\[
\int_{|x-y|< \delta, \;x\neq y} V(x-y)\,(\mu_N-\bar\rho)^{\otimes 2}\geq \frac{1}{N^2}\,\sum_{i\neq j} V_\eps(x_i-x_j)\,\mathbb{I}_{|x_i-x_j|\leq \delta}-C\,\delta^\alpha-\eps.
\]
Observe now that by the second point in  Lemma \ref{lemmatruncate}, we have that
\begin{equation}
\int_{|x-y|\geq \delta} V(x-y)\,(\mu_N-\bar\rho)^{\otimes 2}\geq \int_{|x-y|\geq \delta} V_\eps(x-y)\,(\mu_N-\bar\rho)^{\otimes 2}-C\,\frac{\eps}{\delta^k}.\label{intermed0}
\end{equation}
Therefore by summing, we obtain that
\[
\int_{x\neq y} V(x-y)\,(\mu_N-\bar\rho)^{\otimes 2}\geq \int_{x\neq y} V_\eps(x-y)\,(\mu_N-\bar\rho)^{\otimes 2}-C\,\delta^\alpha-\eps- C\,\frac{\eps}{\delta^k}.
\]
We may simply add the diagonal to find
\begin{equation}
\int_{x\neq y} V(x-y)\,(\mu_N-\bar\rho)^{\otimes 2}\geq \int V_\eps(x-y)\,(\mu_N-\bar\rho)^{\otimes 2}-\frac{C\,\|V_\eps\|_{L^\infty}}{N}-C\,\delta^\alpha-\eps- C\,\frac{\eps}{\delta^k}.\label{intermed1}
\end{equation}
Since $\hat V_\eps\geq 0$, this yields
\[
\int_{x\neq y} V(x-y)\,(\mu_N-\bar\rho)^{\otimes 2}\geq -\frac{C\,\|V_\eps\|_{L^\infty}}{N}-C\,\delta^\alpha-\eps- C\,\frac{\eps}{\delta^k},
\]
and conclude the first point by optimizing in $\eps$ and $\delta$.

\bigskip

\noindent To prove the second point we first remind that 
\[
\frac{1}{N^2}\,\sum_{i\neq j} V(x_i-x_j)\,\mathbb{I}_{|x_i-x_j|\leq \delta}\leq \int_{|x-y|< \delta, \;x\neq y} V(x-y)\,(\mu_N-\bar\rho)^{\otimes 2}+C\,\delta^\alpha, 
\]
and thus
\[\begin{split}
&\frac{1}{N^2}\,\sum_{i\neq j} V(x_i-x_j)\,\mathbb{I}_{|x_i-x_j|\leq \delta}\leq \int_{x\neq y} V(x-y)\,(\mu_N-\bar\rho)^{\otimes 2}+C\,\delta^\alpha\\
&\qquad -\int_{|x-y|\geq \delta, \;x\neq y} V(x-y)\,(\mu_N-\bar\rho)^{\otimes 2}.
\end{split}
\]
Using again \eqref{intermed0}, we get
\[\begin{split}
&\frac{1}{N^2}\,\sum_{i\neq j} V(x_i-x_j)\,\mathbb{I}_{|x_i-x_j|\leq \delta}\leq \int_{x\neq y} V(x-y)\,(\mu_N-\bar\rho)^{\otimes 2}+C\,\delta^\alpha\\
&\qquad -\int_{|x-y|\geq \delta, \;x\neq y} V_\eps(x-y)\,(\mu_N-\bar\rho)^{\otimes 2}+C\,\frac{\eps}{\delta^k}.
\end{split}
\]
Remark that since we are on the torus, $\mathbb{I}_{|x|\geq\delta}$ has $L^1$ norm less than $1$. Therefore the Fourier transform of $V_\eps\,\mathbb{I}_{|x|\geq\delta}$ is dominated by $\hat V_\eps$ and
\[
-\int_{|x-y|\geq \delta, \;x\neq y} V_\eps(x-y)\,(\mu_N-\bar\rho)^{\otimes 2}\leq -\int V_\eps(x-y)\,(\mu_N-\bar\rho)^{\otimes 2}.
\]
Appealing to \eqref{intermed1}, we hence finally get
\[\begin{split}
&\frac{1}{N^2}\,\sum_{i\neq j} V(x_i-x_j)\,\mathbb{I}_{|x_i-x_j|\leq \delta}\leq 2\,\int_{x\neq y} V(x-y)\,(\mu_N-\bar\rho)^{\otimes 2}\\
&\qquad+C\,\delta^\alpha+C\,\frac{\eps}{\delta^k}+\eps +C\,\frac{\|V_\eps\|_{L^\infty}}{N},
\end{split}
\]
which concludes the second point, again by optimizing in $\eps$.
  \end{proof}

\bigskip
 We need to control terms from above like
 \[
 I_N = - \int_{\Pi^{dN}} d\rho_N \int_{\Pi^{2d}\cap \{x\not=y\}}
     \nabla V(x-y) \cdot (\psi(x)-\psi(y)) (d\mu_N-d \bar\rho)^{\otimes 2},
\]
for $\psi$ regular enough in terms of the potential energy. We use Fourier transform for the regularized kernel that not use explicit formula of the kernel as in \cite{Se}, \cite{Du}.  This procedure 
allows  to treat more general kernels because it is not based on the reformulation of the energy in terms of potential or extension representation (for the fractional laplacian) by Caffarelli-Silvestre
as in \cite{Se}, \cite{Se1}, \cite{Du}.  More precisely, we prove that
 \begin{Lemma} 
Let $\psi \in W^{1,\infty}(\Pi^d)$ and if $V$ satisfies \eqref{hyp01}--\eqref{hyp03},
 then for any measure $\nu$, we have that
 \[
-  \int  \nabla V(x-y) (\psi(x)-\psi(y)) \nu^{\otimes 2} \le
     C\, \int |\hat \nu(\xi)|^2\, \hat V(\xi) \, d\xi.
 \]\label{lemfourier}
 \end{Lemma}
\begin{Remark}
Remark that Riesz and Coulomb Kernel satisfy Hypothesis \eqref{hyp01}--\eqref{hyp05}.
\end{Remark}
 
\begin{proof}
 This is where we need \eqref{hyp05} which gives studying $|\xi| \sim |\zeta|$
$$|\hat V(\xi) - \hat V(\zeta)| \le 
        \frac{C |\xi-\zeta|}{1+|\zeta|}  \hat V(\zeta)
$$ 
using Gronwall Lemma. For $\psi$ regular enough,  we use the following calculation
\[
\begin{split}
& - \int \nabla V(x-y) (\psi(x)-\psi(y)) \nu^{\otimes 2}
     = - {\rm Re} \int i (\xi \hat V(\xi) - \zeta \hat V(\zeta)) 
       \hat\psi(\xi-\zeta) \overline{\hat\nu(\xi)} \hat\nu(\zeta)  d\xi d\zeta \\
   &\qquad = - {\rm Re} \int i \left(\xi (\hat V(\xi) -  \hat V(\zeta)) + (\xi-\zeta) \hat V(\zeta)\right) 
       \hat\psi(\xi-\zeta) \overline{\hat\nu(\xi)} \hat\nu(\zeta)  d\xi d\zeta \\
  &\qquad \le C \int  |\xi-\zeta|| \hat \psi(\xi-\zeta)| \sqrt{\hat V(\xi)} \sqrt{\hat V(\zeta)} 
           |\hat\nu(\zeta)| |\hat \nu(\xi)| d\xi d\zeta \\
\end{split}
\]
and therefore, using some regularity on $\psi$, by Cauchy-Schwartz
$$
-  \int \nabla V(x-y) (\psi(x)-\psi(y)) \nu^{\otimes 2} 
  \le C \int |\hat{\nu}(\zeta)|^2  \hat V(\zeta) \, d\zeta.$$
\end{proof}
Of course we cannot directly use Lemma \ref{fourier} with $V$ on $I_N$ as
\[
\int V(x-y)\,(d\mu_N-d\bar\rho)^{\otimes2}
\]
will in general be infinite as the diagonal is not removed. But we can now easily combine Lemma \ref{lemfourier} with Lemma \ref{KNgeq0} to obtain
\begin{Corollary}
  \label{controlINvanishing}
  Assume that $\psi\in W^{s,\infty}$ and that $V$ satisfies \eqref{hyp01}--\eqref{hyp05}. Then there exists a function $\eta(N)$ with $\eta(N)\to 0$ as $N\to \infty$ and such that
  \[
I_N =- \int_{\Pi^{dN}} d\rho_N \int_{\Pi^{2d}\cap \{x\not=y\}}
     \nabla V(x-y) \cdot (\psi(x)-\psi(y)) (d\mu_N-d \bar\rho)^{\otimes 2}\leq C\,\sigma_N\,\mathcal{K}_N(t)+\eta(N).
  \]
  \end{Corollary}
\begin{proof}
  The basic strategy is again to split $I_N$ into two parts
  \[
  \begin{split}
&I_N= - \int_{\Pi^{dN}} d\rho_N \int_{\{|x-y|\leq \delta\}\cap \{x\not=y\}}
    \nabla V(x-y) \cdot (\psi(x)-\psi(y)) (d\mu_N-d \bar\rho)^{\otimes 2}\\
    &\quad -\int_{\Pi^{dN}} d\rho_N \int_{\{|x-y|> \delta\}}
    \nabla V(x-y) \cdot (\psi(x)-\psi(y)) (d\mu_N-d \bar\rho)^{\otimes 2}.
  \end{split}
  \]
  For the second term in the right-hand side, we want to replace $\nabla V$ by $\nabla V_\eps$. We simply use the second point of Lemma \ref{lemmatruncate} again (similarly to the obtention of \eqref{intermed0} in Lemma~\ref{KNgeq0}) to get that
  \begin{equation}
\begin{split}
    & - \int_{\{|x-y|> \delta\}}
  \nabla V(x-y) \cdot (\psi(x)-\psi(y)) (d\mu_N-d \bar\rho)^{\otimes 2}\\
  &\qquad\leq  - \int_{\{|x-y|> \delta\}}
    \nabla V_\eps(x-y) \cdot (\psi(x)-\psi(y)) (d\mu_N-d \bar\rho)^{\otimes 2}+C\,\frac{\eps}{\delta^{k'}}.
\end{split}\label{intermed0'}
  \end{equation}
  For the first term in the right-hand side, we first use the regularity of $\bar\rho$ together with the $L^p$ bound on $V$ from \eqref{hyp01} to deduce that
  \[
  \begin{split}
& - \int_{\{|x-y|\leq \delta\}\cap \{x\not=y\}}
    \nabla V(x-y) \cdot (\psi(x)-\psi(y)) (d\mu_N-d \bar\rho)^{\otimes 2}\\
    &\quad\leq - \int_{\{|x-y|\leq \delta\}\cap \{x\not=y\}}
    \nabla V(x-y) \cdot (\psi(x)-\psi(y)) d\mu_N^{\otimes 2}+C\,\delta^\alpha,
  \end{split}
  \]
  for some $\alpha>0$. We know use the pointwise bound on $\nabla V$ from \eqref{hyp02} and the Lipschitz bound on $\psi$ to obtain that
  \[
  \begin{split}
& -\int_{\{|x-y|\leq \delta\}\cap \{x\not=y\}}
    \nabla V(x-y) \cdot (\psi(x)-\psi(y)) (d\mu_N-d \bar\rho)^{\otimes 2}\\
    &\quad\leq C\,\int_{\{|x-y|\leq \delta\}\cap \{x\not=y\}}
    V(x-y)\, d\mu_N^{\otimes 2}+C\,\delta^\alpha.
  \end{split}
  \]
  
  Using the second point in Lemma \ref{KNgeq0}, we hence have that
\begin{equation}
\begin{split}
& - \int_{\Pi^{dN}} d\rho_N \int_{\{|x-y|\leq \delta\}\cap \{x\not=y\}}
    \nabla V(x-y) \cdot (\psi(x)-\psi(y)) (d\mu_N-d \bar\rho)^{\otimes 2}  \\
&  \hskip9cm   \leq C\,\sigma_N\,\mathcal{K}_N+\eta(N)+\eta(\delta).\label{intermed1'}
 \end{split}
\end{equation}
  By the construction of $V_\eps$ the same estimate applies
  \begin{equation}
  \begin{split}
& -\int_{\Pi^{dN}} d\rho_N \int_{\{|x-y|\leq \delta\}\cap \{x\not=y\}}
    \nabla V_\varepsilon (x-y) \cdot (\psi(x)-\psi(y)) (d\mu_N-d \bar\rho)^{\otimes 2}\\
 & \hskip9cm \leq C\,\sigma_N\,\mathcal{K}_N+\eta(N)+\eta(\delta).\label{intermed2'}
    \end{split} 
    \end{equation}
  We may combine \eqref{intermed2'} with \eqref{intermed0'} to obtain that
  \[
\begin{split}
    &- \int_{\Pi^{dN}} d\rho_N\int_{\{|x-y|> \delta\}\cap \{x\not=y\}}
  \nabla V(x-y) \cdot (\psi(x)-\psi(y)) (d\mu_N-d \bar\rho)^{\otimes 2}\\
  &\qquad\leq - \int_{\Pi^{dN}} d\rho_N\int_{ \{x\not=y\}}
  \nabla V_\eps(x-y) \cdot (\psi(x)-\psi(y)) (d\mu_N-d \bar\rho)^{\otimes 2}\\
  &\qquad\ +C\,\frac{\eps}{\delta^{k'}}+C\,\sigma_N\,\mathcal{K}_N+\eta(N)+\eta(\delta).
\end{split}
 \]
  Together with \eqref{intermed1'}, this finally concludes that
  \begin{equation}
\begin{split}
    &- \int_{\Pi^{dN}} d\rho_N\int_{ \{x\not=y\}}
  \nabla V(x-y) \cdot (\psi(x)-\psi(y)) (d\mu_N-d \bar\rho)^{\otimes 2}\\
  &\qquad\leq - \int_{\Pi^{dN}} d\rho_N\int
  \nabla V_\eps(x-y) \cdot (\psi(x)-\psi(y)) (d\mu_N-d \bar\rho)^{\otimes 2}\\
  &\qquad\ +C\,\frac{\eps}{\delta^{k'}}+C\,\sigma_N\, \mathcal{K}_N+\eta(N)+\eta(\delta).
\end{split}\label{intermed3'}
  \end{equation}
  We know apply Lemma \ref{lemfourier} on $V_\eps$ which by construction still satisfies \eqref{hyp01}-\eqref{hyp05} and this yields
  \[
  \begin{split}
    & - \int_{\Pi^{dN}} d\rho_N\int_{ \{x\not=y\}}
  \nabla V(x-y) \cdot (\psi(x)-\psi(y)) (d\mu_N-d \bar\rho)^{\otimes 2}\\
  &\qquad\leq  C\,\int_{\Pi^{dN}} d\rho_N\int
  V_\eps(x-y)\, (d\mu_N-d \bar\rho)^{\otimes 2}\\
  &\qquad\ +C\,\frac{\eps}{\delta^{k'}}+C\,\sigma_N\, \mathcal{K}_N+\eta(N)+\eta(\delta).
  \end{split}
  \]
  We may now remove the diagonal
   \[
  \begin{split}
    &- \int_{\Pi^{dN}} d\rho_N\int_{ \{x\not=y\}}
  \nabla V(x-y) \cdot (\psi(x)-\psi(y)) (d\mu_N-d \bar\rho)^{\otimes 2}\\
  &\qquad\leq  C\,\int_{\Pi^{dN}} d\rho_N\int_{\{x\neq y\}}
  V_\eps(x-y)\, (d\mu_N-d \bar\rho)^{\otimes 2}+C\,\frac{\|V_\eps\|_{L^\infty}}{N}\\
  &\qquad\ +C\,\frac{\eps}{\delta^{k'}}+C\,\sigma_N\, \mathcal{K}_N+\eta(N)+\eta(\delta),
  \end{split}
  \]
  and using the third point in Lemma \ref{lemmatruncate},
  \[
  \begin{split}
    &- \int_{\Pi^{dN}} d\rho_N\int_{ \{x\not=y\}}
  \nabla V(x-y) \cdot (\psi(x)-\psi(y)) (d\mu_N-d \bar\rho)^{\otimes 2}\\
  &\qquad\leq  C\,\int_{\Pi^{dN}} d\rho_N\int_{\{x\neq y\}}
  V(x-y)\, (d\mu_N-d \bar\rho)^{\otimes 2}+C\,\eps+C\,\frac{\|V_\eps\|_{L^\infty}}{N}\\
  &\qquad\ +C\,\frac{\eps}{\delta^{k'}}+C\,\sigma_N\, \mathcal{K}_N+\eta(N)+\eta(\delta).
  \end{split}
  \]
  The conclusion follows by optimizing in $\eps$ and $\delta$. 
  \end{proof}

\section{Control of terms for Repulsive Kernels with fixed viscosity $\sigma>0$}

 \subsection{A large deviation result}
 \begin{Proposition}\label{LargeDeviation}
There exists $\delta_0>0$ and some exponent $\theta>0$, s.t. for any $\bar\rho\in L^{\infty}\cap \mathcal{P}(\Pi^d)$ with $\log \bar\rho\in W^{1,\infty}$, for any $W\in L^p(\Pi^{2d})$ for some $p>1$ with $W(-x)=W(x)$, $W\geq 0$, $|\nabla W(x)|\leq \frac{C}{|x|^k}$ and finally $\|W\|_{L^1}\leq \delta_0$ then 
  \[
 \frac{1}{N}\log \int_{\Pi^{dN}} e^{-N\,\int_{\{x\neq y\}} W(x-y)\,(d\mu_N-d\bar\rho)^{\otimes^2}}\,\bar\rho^{\otimes^N}\,dX^N\leq  \frac{C}{N^\theta}.
\]
  \end{Proposition}
  \begin{proof}
    We denote
    \[
F(\mu_N)=\int_{\{x\neq y\}} W(x-y)\,(d\mu_N-d\bar\rho)^{\otimes^2},
\]
and introduce a simple truncation $W_\eps$ of $W$ by
\[
W_\eps(x)=W(x)\,\mathbb{I}_{|x|\geq \eps}.
\]
We define as well
\[
F_\eps(\mu)=\int_{ \Pi^{2d} \cap \{x\neq y\}} W_\eps(x-y)\,(d\mu-d\bar\rho)^{\otimes2}.
\]
We may expand $F$ to find
\[\begin{split}
  &F(\mu_N)=\int_{\Pi^{2d}\cap\{x\neq y\}} W(x-y) \,( d \mu_N - d \bar\rho)^{\otimes2}\\
  &\quad  =\frac{1}{N^2} \sum_{ i \ne j } W(x_i - x_j)
   - 2\frac{1}{N}\,\sum_i W\star\bar\rho(x_i)\\
  &\qquad+\int_{\Pi^{2d}\cap\{x\neq y\}} W(x-y)\,\bar\rho(x)\,\bar\rho(y)\,dx\,dy.
\end{split}
\]
Now note that by the $L^p$ integrability on $W$ and the $L^\infty$ integrability on $\bar\rho$, 
\[
\begin{split}
|W_\eps\star\bar\rho(x)-W\star\bar\rho(x)|&\leq \|W_\eps-W\|_{L^1}\,\|\bar\rho\|_{L^\infty}\leq \|W\,\mathbb{I}_{|z|\leq \eps}\|_{L^1}\,\|\bar\rho\|_{L^\infty}\\
&\leq C\,\eps^{1/2p^*}\,\|\bar\rho\|_{L^\infty}.
\end{split}
\]
Further note that since $W\geq 0$, we have that $W_\eps(x)\leq W(x)$ so this directly implies that
\[
-F(\mu_N)\leq -F_\eps(\mu_N)+C\,\|\bar\rho\|_{L^\infty}\,\eps^{1/2p^*}.
\]
We are hence led to bounding
\[
Z_{N,\eps}=\frac{1}{N}\log \int_{\Pi^{dN}} e^{-N\,\int_{\Pi^{2d}} W_\eps(x-y)\,(d\mu_N-d\bar\rho)^{\otimes^2}}\,\bar\rho^{\otimes^N}\,dX^N,
\]
since
\begin{equation}
Z_N=\frac{1}{N}\log \int_{\Pi^{dN}} e^{-N\,\int_{\{x\neq y\}} W(x-y)\,(d\mu_N-d\bar\rho)^{\otimes^2}}\,\bar\rho^{\otimes^N}\,dX^N\leq Z_{N,\eps}+C\,\eps^{1/2p^*}.\label{boundZN}
\end{equation}

We now rely on a quantitative variant of a classical large deviation result
\begin{Theorem}
  Assume that $\log\bar\rho\in W^{1,\infty}$ and that $L$ is a standard convolution kernel. Then there exists a constant $C$ depending only on $d$, $L$,  s.t. for any $F:\;\mathcal{P}(\Pi^d)\to \R$, continuous on continuous functions,  one has that
  \[\begin{split}
  &\frac{1}{N}\log \int_{\Pi^{dN}} e^{-N\,F(L_\delta\star\mu_N)}\,\bar\rho_N\,dX^N\leq I(F)\\
  &\qquad+\frac{C}{N^{1/(d+1)}\,\delta^{d/(d+1)}}\,(\log N+|\log \delta|+\|\log\bar\rho\|_{L^\infty})+C\,\delta\,\|\log\bar \rho\|_{W^{1,\infty}},
  \end{split}
  \]
  where
  \begin{equation}
I(F)=\max_{\mu\in \mathcal{P}(\Pi^d)}  -\bigl[F(\mu)+\int_{\Pi^d} \mu\,\log \frac{\mu}{\bar\rho}\,dx\Bigr]. \label{largedeviationfunct}
  \end{equation}
  \label{variantlargedeviation}
  \end{Theorem}
  \begin{proof}
The proof of Theorem \ref{variantlargedeviation} relies on classical arguments and we refer to our upcoming article for more details.
Since $|\nabla W|\leq C/|x|^k$, the potential $W_\eps$ is smooth and hence
\[
-F_\eps(\mu_N)\leq -F_\eps(L_\delta\star \mu_N)+C\,\frac{\delta}{\eps^k}.
\]
Using Theorem \ref{variantlargedeviation}, we thus have that
\begin{equation}
\begin{split}
  Z_{N,\eps}\leq &I(F_\eps)+C\,\frac{\delta}{\eps^k}+\frac{C}{N^{1/(d+1)}\,\delta^{d/(d+1)}}\,(\log N+|\log \delta|+\|\log\bar\rho\|_{L^\infty})\\
  &+C\,\delta\,\|\log\bar \rho\|_{W^{1,\infty}}.
\end{split}\label{boundZNeps}
\end{equation}
The last step of the proof is hence to estimate $I(F_\eps)$ for which we appeal to
\begin{Lemma}
    For any $\bar\rho\in L^\infty$, there exists a truncation $\delta_0$ s.t.  for any $\tilde W$ with $\|\tilde W\|_{L^1}\leq \delta_0$ and $W(x)\geq 0$ then, defining
    \[
F_{\tilde W}(\mu)=\int_{\{x\neq y\}} \tilde W(x-y)\,(\mu(dx)-\bar\rho(x)\,dx)\,(\mu(dy)-\bar\rho(y)\,dy),
  \]
  one then has that $I(F_{\tilde W})=0$.\label{I(F)log} 
\end{Lemma}
Assuming that Lemma \ref{I(F)log} holds, we may apply it to $\tilde W=W_\eps$ which implies  $I(F_\eps)=0$ and finally combining \eqref{boundZN} and \eqref{boundZNeps}
\[
\begin{split}
  Z_{N}\leq &C\,\eps^{1/2p^*}+C\,\frac{\delta}{\eps^k}+\frac{C}{N^{1/(d+1)}\,\delta^{d/(d+1)}}\,(\log N+|\log \delta|+\|\log\bar\rho\|_{L^\infty})\\
  &+C\,\delta\,\|\log\bar \rho\|_{W^{1,\infty}}.
\end{split}
\]
We may immediately conclude by optimizing in $\eps$ and $\delta$.
  \end{proof}

  \begin{proof}[Sketch of the proof of Lemma \ref{I(F)log}]
    Since $\tilde W\geq 0$, $I(F_{\tilde W})$ is coercive and by considering a maximizing sequence, we may find a maximum $\bar\mu$ which is bounded in  $L\,\log L$. Such a maximum must satisfy that
    \[
1+\log \frac{\bar\mu}{\bar\rho} + {2}\,\tilde W\star(\bar\mu-\bar\rho)=\kappa,
\]
where the constant $\kappa$ is chosen so that $\int \bar\mu=1$. This may be rewritten as
\[
\bar\mu=\frac{\bar\rho}{M}\,e^{ - {2}\,\tilde W\star(\bar\mu-\bar\rho)},\quad M=\int \bar\rho\,e^{ - {2\,}\tilde W\star(\bar\mu-\bar\rho)}\,dx.
\]
Let us denote $u= - \tilde W\star (\bar\mu-\bar\rho)$ and to emphasize the dependence on $u$ in $M$
\[
M=M_u=\int \bar\rho\,e^{{2}\,u(x)}\,dx.
\]
We observe that $u$ is a solution to
\begin{equation}
  u= - \tilde W\star\left(\bar\rho\,\left(\frac{e^{{2}\,u(x)}}{M_u}-1\right)
  \right),\label{eulerlagrange}
  \end{equation}
which is in fact a sort of non-linear elliptic equation. It is straightforward to show that the unique solution to \eqref{eulerlagrange} is $u=0$ provided that $\|\tilde W\|_{L^1}$ is small enough.
    \end{proof}
\end{proof}

  \subsection{Control on $E_N$}
  %
  The second ingredient to bound $E_N$ from below is the following classical convexity inequality
  \begin{equation}\label{ineg}
\int_{\Pi^{dN}} \psi(X^N)\,d\rho_N\leq \frac{1}{\alpha} \,  \frac{1}{N}\int d\rho_N\,\log \frac{\rho_N}{\bar\rho_N}+ \frac{1}{\alpha} \frac{1}{N}\log \int_{\Pi^{dN}} e^{ \alpha   N\,\psi(X^N)}\,d\bar\rho_N. 
  \end{equation}
  The proper control of $E_N$ however requires truncating interactions after some distance so that we define
  \[
\mathcal{K}_N^\eta=\frac{1}{2\,\sigma}\,\int_{\Pi^{d\,N}} \int_{\{x\neq y\}} V(x-y)\,\chi(|x-y|/\eta)\,(d\mu_N-d\bar\rho)^{\otimes2}\,\rho_N\,dX^N,
\]
where $\chi$ is some smooth non-negative function with $\chi$=1 on $[0, 1]$ and $\chi=0$ on $[2,\infty)$, together with $\hat\chi\geq 0$. Finally we define 
$E_N^{\eta}=\mathcal{H}_N+\mathcal{K}_N^\eta$. 
  Combined with Prop. \ref{LargeDeviation} this inequality lets us bound from below $E_N^\eta$ as per
  \begin{Proposition}
    \label{boundEN}
 Assume that $V$ satisfies \eqref{hyp01}-\eqref{hyp02} and \eqref{hyp06} then there exists $\eta_0>0$ and $\theta>0$ so that for any $\eta\leq \eta_0$  
  \[
E_N^\eta\geq -\frac{C}{N^\theta}.
\]
Moreover for any $\delta\leq\eta$ with $\eta\leq \eta_0$
\[
\mathbb{E}\left(\frac{1}{N^2}\,\sum_{i\neq j} V(x_i-x_j)\,\mathbb{I}_{|x_i-x_j|\leq \delta}\right)\leq E_N^\eta +\frac{C}{N^\theta}+C\,\delta^\theta.
\]
\end{Proposition}   
  \begin{proof}
    For the first point, using \eqref{ineg} on $\mathcal{K}_N^\eta$, we find that
    \[
-\mathcal{K}_N^\eta\leq H_N+\frac{1}{N}\log \int_{\Pi^{dN}} e^{-\frac{N}{2\,\sigma}\,\int_{\{x\neq y\}} V(x-y)\,\chi(|x-y|/\eta)\, (d\mu_N-d\bar\rho)^{\otimes2}}\,d\bar\rho_N.
\]
We now apply Prop. \ref{LargeDeviation} to $W=\frac{1}{2\,\sigma}\,V(x)\,\chi(|x|/\eta)$. From \eqref{hyp01} and \eqref{hyp001}, $W$ trivially satisfies all assumptions of Prop. \ref{LargeDeviation} with the exception of $\|W\|_{L^1}\leq \delta_0$. For this, we remark that
\[
\left\|V(x)\,\chi(|x|/\eta)\right\|_{L^1}\leq \int_{|x|\leq 2\,\eta} V(x)\,dx\leq C\,\|V\|_{L^p}\,\eta^{1/p^*},
\]
so that $\|W\|_{L^1}\leq \delta_0$ is ensured by choosing $\eta\leq\eta_0$ with $\eta_0$ small enough. Therefore Prop. \ref{LargeDeviation} implies that
\[
-\mathcal{K}_N^\eta\leq H_N+\frac{C}{N^\theta}.
\]
For the second point, observe first that
\[\begin{split}
\frac{1}{N^2}\,\sum_{i\neq j} V(x_i-x_j)\,\mathbb{I}_{|x_i-x_j|\leq \delta}&\leq \frac{1}{N^2}\,\sum_{i\neq j} V(x_i-x_j)\,\chi(|x_i-x_j|/ \delta)\\
&\leq \int_{\{x\neq y\}} V(x-y)\,\chi(|x-y|/ \delta)\,(d\mu_N-d\bar\rho_N)^{\otimes2}
+C\,\delta^\theta.
\end{split}
\]
Therefore
\[\begin{split}
&\mathbb{E}\left(\frac{1}{2\,\sigma\,N^2}\,\sum_{i\neq j} V(x_i-x_j)\,\mathbb{I}_{|x_i-x_j|\leq \delta}\right)
-E_N^\eta \\
&\quad \leq -H_N-\frac{1}{2\,\sigma}\int_{\{x\neq y\}} V(x-y)\,(\chi(|x-y|/\eta)-\chi(|x-y|/ \delta))\,(d\mu_N-d\bar\rho_N)^{\otimes2}+C\,\delta^\theta.
\end{split}
\]
Using \eqref{ineg}
\[\begin{split}
&
\mathbb{E}\left(\frac{1}{2\,\sigma\,N^2}\,\sum_{i\neq j} V(x_i-x_j)\,\mathbb{I}_{|x_i-x_j|\leq \delta}
\right) -E_N^\eta \\
&\quad \leq \frac{1}{N}\log\int_{\Pi^{d\,N}} e^{-\frac{N}{2\,\sigma}\int_{\{x\neq y\}} V(x-y)\,(\chi(|x-y|/\eta)-\chi(|x-y|/ \delta))\,(d\mu_N-d\bar\rho_N)^{\otimes2}}\,d\bar\rho_N+C\,\delta^\theta.
\end{split}
\]

Of course $\chi(|x-y|/\eta)-\chi(|x-y|/ \delta)\geq 0$ if $\delta\leq \eta$ so that we may again apply Prop. \ref{LargeDeviation} to $W(x)=V(x)\,(\chi(|x|/\eta)-\chi(|x|/ \delta))$ and find as claimed
\[
\mathbb{E}\left(\frac{1}{2\,\sigma\,N^2}\,\sum_{i\neq j} V(x_i-x_j)\,\mathbb{I}_{|x_i-x_j|\leq \delta}\right) -E_N^\eta
\leq C\,\delta^\theta+\frac{C}{N^\theta}.
\]
  \end{proof}
  \subsection{Control of the right-hand side}
  As for the case $\sigma\to 0$, the goal is to control
  \[
I_N= -\int_{\Pi^{d\,N}}\int_{\{x\neq y\}} \nabla V(x-y)\cdot(\psi(x)-\psi(y))\,(d\mu_N-d\bar\rho)^{\otimes2}\,d\rho_N.
  \]
  The Fourier assumption \eqref{hyp05} for $\sigma>0$ is more general that assumption \eqref{hyp05} in the case $\sigma=0$ because we can use Inequality \eqref{ineg} and Theorem \ref{LargeDeviation}
to control the new term in the Fourier procedure. More precisely, we have
\begin{Lemma} Assume that $\psi \in W^{k,+\infty}$ with $k$ large enough and that $V$ satisfies  \eqref{hyp01}-\eqref{hyp02} and
  \begin{equation}
|\hat V(\xi) - \hat V(\zeta)| \le C \frac{|\xi-\zeta|}{1+|\zeta|} \Bigl(\hat V(\zeta) + \frac{\chi(\zeta)}{1+|\zeta|^{d-\alpha}}\Bigr)\label{diffhatV}
  \end{equation}
with $0<\alpha<d$ and $\chi \in L_\zeta^\infty$. Then for any measure $\nu$, we have that
\begin{equation} \label{fourier}
- \int \nabla V(x-y) (\psi(x)-\psi(y)) \nu^{\otimes 2} \le
     C\, \int |\hat \nu(\xi)|^2\, (\hat V(\xi) + \chi(\xi)/(1+|\xi|^{d-\alpha})  \, d\xi.
 \end{equation}\label{lemfourier2}
\end{Lemma}
\begin{proof} 
The proof follows the same lines as for Lemma \ref{lemfourier}. 
\end{proof}
Note that the extra term in the right-hand side of inequality \eqref{fourier} is controlled by the relative entropy because going back to the physical space, we will get a term written as
$$\int_{x\neq y} G(x-y)  (\mu_N-\bar\rho)^{\otimes 2} (dx dy)
$$
with $G(x)\sim |x|^{-\alpha} \in L^p$ for some $p>1$. 
\noindent We can now prove the equivalent of Corollary \ref{controlINvanishing}.
\begin{Corollary}
  \label{controlINnonvanishing}
  Assume that $\psi\in W^{s,\infty}$ and that $V$ satisfies all assumptions \eqref{hyp01}--\eqref{hyp06}. Then there exists $\eta_0$ and $\theta>0$ s.t. if $\eta\leq \eta_0$
  \[
I_N = -\int_{\Pi^{dN}} d\rho_N \int_{\Pi^{2d}\cap \{x\not=y\}}
     \nabla V(x-y) \cdot (\psi(x)-\psi(y)) (d\mu_N-d \bar\rho)^{\otimes 2}\leq C\,E_N^\eta(t)+\frac{C}{N^\theta}.
  \]
  \end{Corollary}
\begin{proof}
  The proof closely follows the one for Corollary \eqref{controlINvanishing}. More precisely all arguments up to obtaining \eqref{intermed3'} are identical and hence we have still
\[
  \begin{split}
I_N &= -\int_{\Pi^{dN}} d\rho_N \int_{\{x\not=y\}}
\nabla V(x-y) \cdot (\psi(x)-\psi(y)) (d\mu_N-d \bar\rho)^{\otimes 2}\\
&\leq -\int_{\Pi^{dN}} d\rho_N \int_{\{x\not=y\}}
\nabla V_\eps(x-y) \cdot (\psi(x)-\psi(y)) (d\mu_N-d \bar\rho)^{\otimes 2}\\
&\ +C\,\frac{\eps}{\delta^{k'}}+C\,\mathcal{K}_N +\frac{C}{N^\theta}+C\,\eta^\theta.
  \end{split}
  \]
  From this point, instead of Lemma \ref{lemfourier}, we apply Lemma \ref{lemfourier2}. Let us take $V_\eps$ as constructed in  Lemma \ref{lemmatruncate}. From Assumptions \eqref{hyp05}-\eqref{hyp06}, we have that
  \[
  |\hat V(\xi) - \hat V(\zeta)| \le C \frac{|\xi-\zeta|}{1+|\zeta|} \Bigl(\hat V(\zeta) + \frac{\hat K_\eps(\zeta)}{1+|\zeta|^{d-\alpha} }\Bigr),
  \]
  so that \eqref{diffhatV} is satisfied with $\chi=\hat K_\eps$ and Lemma \ref{lemfourier2} yields
\[
  \begin{split}
I_N &\leq C\,\int_{\Pi^{dN}} d\rho_N \int_{\{x\not=y\}}
V_\eps(x-y)\, (d\mu_N-d \bar\rho)^{\otimes 2}\\
&\ +C\,\int_{\Pi^{dN}} d\rho_N \int \frac{|\hat\mu_N-\hat{\bar\rho}|^2}{1+|\xi|^{d-\alpha}}\,\hat K_\eps(\xi)\,d\xi\\ 
&\ +C\,\frac{\eps}{\delta^{k'}}+C\,\mathcal{K}_N +\frac{C}{N^\theta}+C\,\eta^\theta.
  \end{split}
  \]
  Because we have that $V_\eps\leq C\,V + \varepsilon$, by Lemma \ref{lemmatruncate}, we deduce that
  \[
  \begin{split}
I_N &\leq C\,\int_{\Pi^{dN}} d\rho_N \int_{\{x\not=y\}}
V(x-y)\, (d\mu_N-d \bar\rho)^{\otimes 2}\\
&\ +C\,\int_{\Pi^{dN}} d\rho_N \int \frac{|\hat\mu_N-\hat{\bar\rho}|^2}{1+|\xi|^{d-\alpha}}\,\hat K_\eps(\xi)\,d\xi\\ 
&\ +C\,\frac{\eps}{\delta^{k'}}+C\,\mathcal{K}_N +\frac{C}{N^\theta}+C\,\eta^\theta.
  \end{split}
  \]
  Using Prop. \ref{boundEN}, we then have that
  \begin{equation}
  \begin{split}
I_N &\leq C\,E_N^{\eta} +C\,\int_{\Pi^{dN}} d\rho_N \int \frac{|\hat\mu_N-\hat{\bar\rho}|^2}{1+|\xi|^{d-\alpha}}\,\hat K_\eps(\xi)\,d\xi\\ 
&\ +C\,\frac{\eps}{\delta^{k'}}+C\,\mathcal{K}_N +\frac{C}{N^\theta}+C\,\eta^\theta.
  \end{split}\label{intermed0''}
  \end{equation}
  By taking the inverse Fourier transform, we can write
  \[
  \begin{split}
& \int_{\Pi^{dN}} d\rho_N \int \frac{|\hat\mu_N-\hat{\bar\rho}|^2}{1+|\xi|^d}\,K_\eps(\xi)\,d\xi=\int_{\Pi^{dN}} d\rho_N \int_{\Pi^{2d}}
K_\eps\star G(x-y)\, (d\mu_N-d \bar\rho)^{\otimes 2},
    \end{split}
  \]
  where $\hat G(\xi)=(1+|\xi|)^{-d+\alpha}$ and hence $G(x)\leq C\,\frac{1}{|x|^\alpha}$. We may remove the diagonal
  \[
  \begin{split}
    & \int_{\Pi^{dN}} d\rho_N \int \frac{|\hat\mu_N-\hat{\bar\rho}|^2}{1+|\xi|^d}\,K_\eps(\xi)\,d\xi\\
    &\qquad\leq\int_{\Pi^{dN}} d\rho_N \int_{\{x\neq y\}}
K_\eps\star G(x-y)\, (d\mu_N-d \bar\rho)^{\otimes 2}+C\,\frac{1}{\eps^\alpha N},
    \end{split}
  \]
  and then simply bound $K_\eps\star G$ by $1/|x|^\alpha$ to get
  \[
  \begin{split}
    & \int_{\Pi^{dN}} d\rho_N \int \frac{|\hat\mu_N-\hat{\bar\rho}|^2}{1+|\xi|^d}\,K_\eps(\xi)\,d\xi\\
    &\qquad\leq C\,\int_{\Pi^{dN}} d\rho_N \int_{\{x\neq y\}}
      \frac{1}{|x-y|^\alpha}\, (d\mu_N-d \bar\rho)^{\otimes 2}+C\,\frac{1}{\eps^\alpha N}.
    \end{split}
  \]
  By using the convexity inequality and Prop. \ref{LargeDeviation}, we finally obtain that
 \[
  \begin{split}
    & \int_{\Pi^{dN}} d\rho_N \int \frac{|\hat\mu_N-\hat{\bar\rho}|^2}{1+|\xi|^d}\,K_\eps(\xi)\,d\xi\leq C\,\mathcal{H}_N+C\,\frac{1}{\eps^\alpha N},
    \end{split}
  \]
  and inserting this into \eqref{intermed0''}
\[
  \begin{split}
I_N &\leq C\,E_N^{\eta} +C\,\mathcal{H}_N+C\,\frac{1}{\eps^\alpha N} +C\,\frac{\eps}{\delta^{k'}}+C\,\mathcal{K}_N+\frac{C}{N^\theta}+C\,\eta^\theta,
  \end{split}
  \]
  concluding by optimizing on $\delta$ and $\eps$. 
\end{proof}
\section{Concluding the proofs of Theorems \ref{Conv} and \ref{Convgene}}
The proofs of Theorems \ref{Conv} and \ref{Convgene} now follow from a straightforward Gronwall argument starting from \eqref{INEG}. For the proof of Theorem \ref{Conv}, we rescale \eqref{INEG} to deduce
\[
\sigma_N\,E_N(t)\leq \sigma_N E_N(t=0)+\int_0^t I_N(s)\,ds,
\]
with
\[
I_N=\int_{\Pi^{d\,N}} \int_{\{x\neq y\}} \nabla V(x-y)\cdot (\psi(x)-\psi(y))\, (d\mu_N-d\bar\rho)^{\otimes2}\,d\rho_N,
\]
and
\[
\psi(x)=\frac{\sigma_N}{2}\,\nabla \log \frac{\bar\rho}{G_{\bar\rho}}(x).
\]
Note that $\psi\in W^{s,\infty}$ uniformly in $N$ since $\sigma_N\,\log G_{\bar\rho}$ is smooth uniformly in $N$ (but the scaling by $\sigma_N$ is needed here). We may hence apply Corollary \ref{controlINvanishing} to obtain that
\[
I_N\leq C\,\sigma_N\,\mathcal{K}_N(t)+\eta(N),
\]
and hence recalling that $E_N=\mathcal{K}_N(t)+\mathcal{H}_N(t)$,
\begin{equation}
\sigma_N\,\mathcal{K}_N(t)+\sigma_N\,\mathcal{H}_N(t)\leq \sigma_N\,\mathcal{K}_N(t=0)+\sigma_N\,\mathcal{H}_N(t=0)+C\,\int_0^t \sigma_N\,\mathcal{K}_N(s)\,ds+\eta(N). \label{gronwall0}
  \end{equation}
We recall from Lemma \ref{KNgeq0} that $\sigma_N\,\mathcal{K}_N\geq -\eta(N)$ and hence applying Gronwall lemma, we conclude from \eqref{gronwall0} that
\[
\sigma_N\,\mathcal{K}_N(t)\leq e^{C\,t}\,\left(\sigma_N\,\mathcal{K}_N(t=0)+\sigma_N\,\mathcal{H}_N(t=0)+\eta(N)\right),
\]
finishing the proof of Theorem \ref{Conv}. For the proof of Theorem \ref{Convgene}, as mentioned
just after \eqref{ineg}, the control of $E_N$ required truncating interactions after some distance and
it remains now to control the long range part in $V$ which we need to deal with on it own. The procedure is well explained in \cite{BrJaWa1}.  First we recall that this long-range part reads
$$W(x)= V(x) (1- \chi(|x|/\eta))$$
and we define ${\mathcal K}_N^W$ given by \eqref{modulatedenergy} with  $G_N^W$ and $G_{\bar\rho_N}^W$ replacing $V$
 by $W$ in \eqref{GN} and \eqref{G}.
The different results in Section 6 actually concern $E_N^\eta$ and  we need to evaluate the contribution
of ${\mathcal K}_N^W$ which is the complement. Calculating its evolution in time and using the fact that the 
kernel $V$ is $W^{2,\infty}$ far from 0, we can prove that
$$\frac{d}{dt} {\mathcal K}_N(G_N^W\vert G_{\bar\rho_N}^W) \le C {\mathcal H}_N(\rho_N\vert\bar\rho_N) + \frac{C}{N}.
$$
The interested readers are referred to \cite{BrJaWa} for more explanations and to \cite{BrJaWa1} for the forthcoming single complete document.
Plugging everybody together allows to get Theorem \ref{Convgene}.
%
\section{An attractive interesting case: The Patlak-Keller-Segel Kernel.}
In space dimension $2$, the Patlak-Keller-Segel system
reads 
    \[
   \left\{ \begin{split}
      &\partial_t \bar\rho+\mbox{div}\,(\bar\rho\,u)=\sigma\,\Delta\,\bar\rho,\\
      &u=\nabla\Phi,\quad -\Delta\Phi=2\,\pi\,\lambda\,\bar\rho.
      \end{split}\right.
    \]
This is an important system in biology for instance. Classical solutions for such system may not exist for all times as the singular attractive interactions can lead to concentration (see J. Dolbeault and B. Perthame \cite{DoPe}):
  \begin{itemize}
  \item Global existence of classical solution  if $\lambda\leq 4\,\sigma$ (or $\lambda\leq 2\,d\,\sigma$).
    \item Always blow-up if $\lambda>4\,\sigma$.
    \end{itemize}
 Based on the free energy of the system
    \[
\int \bar\rho\,\log\bar\rho\,dx+\frac{\lambda}{2}\,\int \log |x-y|\,\bar\rho(x)\,\bar\rho(y)\,dx\,dy.
    \]
Blanchet-Dolbeault-Perthame (see \cite{BlDoPe}) show  existence of  global weak satisfying free energy control with subcritical mass. Note that our modulated free energy may be seen as a particle version 
of such free energy.  Particle approximation of the Patlak-Keller-Segel system has been studied by several authors such as \cite{CaPe}, \cite{GoQu}, \cite{HaSc} for example. Recently N.~Fournier and B. Jourdain (see \cite{FoJo}) proved  limit for  $\lambda<\sigma$ with no quantitative estimates. In all these papers, the particle system is also studied from an existence view point which is an important and difficult question. This is not the objective of our study which focuses on quantitative estimates under assumptions of existence of solutions.

 To prove a quantitative estimate between the particle approximation of the Patlak-Keller-Segel 
 system and its formal limit, the upper-bound of \eqref{IN} in the inequality \eqref{INEG} encoding the 
 propagation of $E_N$ is more simple than for the repulsive case. It uses that $|\nabla V(x)| \le C/|x|$ and Theorem \cite{JaWa}.
 The lower bound control of $ {\mathcal K}_N$ is more complicated: We choose to prove an upper bound of the 
 opposite. The method is quite similar that the repulsive case when $\sigma>0$ is fixed. It uses an appropriate cut-off smooth function close to singularity with a regularization of the kernel and the proof that for an appropriate cut-off size the large deviation function is zero. This uses the Logarithmic Hardy-Littlewood-Sobolev inequality to show the maximum is attained for $\eta$ small
 enough for $\mu= \bar\rho$.  The interested readers are referred to \cite{BrJaWa} for more 
 explanations and to \cite{BrJaWa1} for a single complete document.

\section{Conclusion.}
Using the right physics is the key in \cite{BrJaWa}  to make the link between two important
results namely  \cite{Se} and \cite{JaWa}: This link allows to consider more general 
singular kernels with possible presence of viscosity. The method provides a statistical 
control with  a large class of attractive-repulsive interactions  but some works are needed
 to obtain the best convergence rate and it is not yet fully clear how general the interactions
can be. Note that we have not used the presence of the diffusive term \eqref{diffusive} in the inequality
concerning the free-energy which could perhaps help to improve the rate of convergence
when $\sigma>0$ is fixed.
   It could be also interesting to study the case with blow-up for attractive kernels namely  
the super-critical cases. An other important problem could be non-gradient flow systems
and hamiltonian systems.  It could also be the extension of the work to the Keller-Segel
parabolic-parabolic equations, see \cite{To}. An existence result improving the results 
by \cite{CaPe},  \cite{FoJo} for the particle approximation of the Patlak-Keller-Segel
 system is also a challenging and interesting problem. 

\bigskip

\noindent {\bf Acknowledgments.} The first author wants to thank F. Golse and F.  Merle for the invitation to present the work \cite{BrJaWa} at IH\'ES for the Laurent Schwartz seminar on November 2019  and for questions which have been the starting point to design the content of the document to help readers to understand the main steps of this new approach. The second author wants to thank  S. Serfaty for sharing many insights on her result. The two first authors thank also L. Saint-Raymond for discussions on many-particle systems and their mean field limit. The first author is partially supported by the SingFlows project, grant ANR-18-CE40-0027. The second author is  partially supported by NSF DMS Grant 161453, 1908739, and NSF Grant RNMS (Ki-Net) 1107444.

\end{document}